\documentclass[a4paper, 11pt]{article}

\usepackage{graphicx}
\usepackage{amssymb,amsmath,amscd,amsthm,latexsym}
\usepackage{tabularx}
\usepackage{array}
\usepackage{delarray}
\DeclareGraphicsExtensions{.jpg,.pdf,.mps,.eps,.png}
\usepackage[utf8x]{inputenc}
\usepackage{ucs}
\usepackage{amsfonts}
\usepackage{multicol}
\usepackage{dsfont}
\usepackage{mathrsfs}
\usepackage{epsfig}
\usepackage[T1]{fontenc}
\usepackage{slashbox}
\usepackage[center]{caption}
\usepackage{color}

\input xy
\xyoption{all}
\usepackage{xy}

\newtheorem{thm}{Théorème}
\newtheorem{lem}{Lemme}
\newtheorem{prop}{Proposition}
\newtheorem{cor}{Corollaire}
\newtheorem{propriete}{Propriété}
\newtheorem{defi}{Définition}

\newcommand{\R}{\mathbb{R}}

\newcommand{\PP}{\mathbb{P}}

\begin{document}

\title{Goodness of fit statistics for sparse contingency tables}
\author{Audrey Finkler}
\date{}

\maketitle

%\tableofcontents

\numberwithin{equation}{section}
\setcounter{secnumdepth}{3}

\begin{abstract}
Statistical data is often analyzed as a contingency table, sometimes with empty cells called zeros.
Such sparse tables can be due to scarse observations classified in numerous categories, as for example in genetic association studies.
Thus, classical independence tests involving Pearson's chi-square statistic $Q$ or Kullback's minimum discrimination information statistic $G$ cannot be applied because some of the expected frequencies are too small.
%the minoring conditions on the expected frequencies are likely not to be satisfied.
More generally, we consider goodness of fit tests with composite hypotheses for sparse multinomial vectors and suggest simple corrections for $Q$ and $G$ that improve and generalize known procedures such as Ku's.
We show that the corrected statistics share the same asymptotic distribution as the initial statistics.
We produce Monte Carlo estimations for the type I and type II errors on a toy example.
Finally, we apply the corrected statistics to independence tests on epidemiologic and ecological data.
\end{abstract}

\section{Introduction and notations} \label{sec1}

Physical, sociological or biological surveys often lead to data presented as contingency tables.
These surveys aim at studying relationships such as total, mutual, partial or conditional independence between several characters in a population. 
Table notations are very useful to formulate the test and make the corresponding hypotheses explicit, but can be cumbersome when the number of characters exceeds three.
For theoretical results, we therefore use vector notations instead, and we reformulate the independence test as a multinomial goodness of fit test.

\subsection{Goodness of fit tests}

Let $p=(p_1,\ldots,p_R)$ be a probability vector of dimension $R$ where $R \geqslant 2$ is the total number of cross-classifying categories.
Let $n$ be the sample size and $x=(n_1,\ldots,n_R)$ the vector of observed frequencies, realization of $X=(N_1,\ldots,N_R)$ with distribution $\mathcal M(n;p)$, multinomial distribution with parameters $n$ and $p$.
We denote by $p^0$ the probability vector under the null hypothesis and consider the following test :
\begin{equation}
{\cal H}_0 :\quad p=p^0 \quad \text{ against } \quad {\cal H}_1 :\quad p \neq p^0.
\end{equation}

Popular goodness of fit statistics include Pearson's chi-square statistic $Q$ and Kullback's minimum discrimination information statistic $G$, defined in~\cite{r2,r1}.
For two probability distributions $p$ and $p'$, these are written :
\begin{equation}
Q_p(p')=n \sum_{r=1}^{R} \frac{(p'_r-p_r)^2}{p_r},
\end{equation}
and
\begin{equation} \label{Ggene}
G_p(p')=2n \sum_{r=1}^{R} p'_r \ln \frac{p'_r}{p_r}.
\end{equation}
They belong to the power divergence statistics family $\{ RC^{\lambda}, \lambda \in \R\}$ defined by Read and Cressie in \cite{r3}, respectively for $\lambda=1$ and $\lambda \to 0$, where :
\begin{equation}
RC^{\lambda}_p(p')=\frac{2n}{\lambda (\lambda + 1) } \sum_{r=1}^R p'_r \left[ \left( \frac{p'_r}{p_r} \right)^{\lambda}-1 \right].
\end{equation}
For a review on goodness of fit testing methods and statistics, see \cite{r10}.

The vector $p^0$ is not always completely specified.
We assume that $p^0$ is a known function of an unknown parameter $\theta$ of $\Theta \subseteq \R^s$ with $s<R-1$, and denote $p^0=p^0(\theta)$ with $\theta=(\theta_1,\ldots,\theta_s)$.
We suppose that the functions $\theta \mapsto p^0(\theta)$ we consider here are bijective and estimate $p^0(\theta)$ by $p^{*0}=p^0(\theta^*)$ where $\theta^*$ is the maximum likelihood estimator of $\theta$.
As the probability $p$ is generally unknown, we also estimate the $p_{r}$ for $r$ in $\{1,\ldots,R\}$ by their maximum likelihood estimators $p^*_r=n_{r}/n$.
Throughout this paper, underlying indexes $n$ are omitted for simplicity of notation, and we consider the statistics $Q_{p^{*0}}(p^*)$ and $G_{p^{*0}}(p^*).$

Read and Cressie show that under Birch's regularity conditions, see \cite{r11}, the statistics $RC^{\lambda}$ are asymptotically equivalent and share a common chi-square limit distribution:
\begin{thm} \label{asymp}
\begin{equation} 
\forall \lambda \in \R, \quad \lim_{n \to +\infty} \mathcal L_{{\cal H}_0}(RC^{\lambda}_{p^{*0}}(p^*))=\chi^2_{R-s-1}.
\end{equation}
\end{thm}

This implies the following result :

\begin{cor}
\begin{equation} 
\lim_{n \to +\infty} \mathcal L_{{\cal H}_0}(Q_{p^{*0}}(p^*))=\lim_{n \to +\infty} \mathcal L_{{\cal H}_0}(G_{p^{*0}}(p^*))=\chi^2_{R-s-1}.
\end{equation}
\end{cor}

This asymptotic result is a consequence of the Central Limit Theorem.
Classical empirical limitations include that the sample size $n$ must be over~$30$ and that all expected frequencies must be over $5$.
Usually $G_{p^{*0}}(p^*)$ is preferred to $Q_{p^{*0}}(p^*)$ because it is less sensitive to small cell frequencies.
Read and Cressie recommend the use of $RC^{2/3}_{p^{*0}}(p^*)$ instead for $n \geqslant 10$ and a minimum expected frequency over $1$.
Sometimes, even the lower bound $0.5$ is accepted for the expected frequencies.
A review on these conditions can be found in \cite{r4}.

\subsection{Sparse tables}

When there are too few subjects in the study or when the classifying categories are too numerous, the table comprises one or several empty cells called random zeros.
The table is then called sparse and it is likely that at least one cell has an expected frequency below $0.5$.
%, so that Theorem \ref{asymp} does not apply.
Random zeros would not appear if the sample was of sufficient size.
Structural zeros, corresponding to cells with an expected probability of zero, are not considered here and should be suppressed.
We therefore assume that the following condition is satisfied :
\begin{equation} \label{nonnul}
p^{*0}_r \neq 0, \: r \in \{1,\ldots,R\}.
\end{equation}

Sparse tables can nonetheless be tested for independence, for example by regrouping cells so that the condition on the expected frequencies is satisfied.
However this procedure is not always data relevant.
Fisher's exact test given in \cite{r5} applies without restrictions, except that it becomes numerically unmanageable when the table dimension grows.
This leads to the use of Monte Carlo simulation methods as explained in \cite{r6}.

The approach we propose here consists in correcting the historical statistics $Q_{p^{*0}}(p^*)$ and $G_{p^{*0}}(p^*)$ according to the number of zero cells, by generalizing and improving a method designed by Ku in \cite{r7}.  

\section{Corrections for Pearson's and Kullback's statistics}

Let $C$ be the random variable giving the number of zeros in the vector or the contingency table, and $c$ a realization of $C$ with $R-c \geqslant 1$.
For simplicity, we assume that $n_1=n_2=\dots=n_c=0$ and that $n_j \geqslant 1$ for all $ j$ in $ \{c+1,\ldots,R\}$. 
The maximum likelihood estimator $p^*=(0,\ldots,0,n_{c+1}/n,\ldots,n_R/n)$ underestimates the $p_i$ for $i$ in $\{1,\ldots,c\}$ and overestimates the $p_j$ for $j$ in $\{c+1,\ldots,R\}$.
Its use when $c \neq 0$ thus has consequences on the statistics $Q_{p^{*0}}(p^*)$ and $G_{p^{*0}}(p^*)$.

\subsection{Ku's correction for one zero}

Ku argues in \cite{r7} that $G_{p^{*0}}(p^*)$ tends to inflate with respect to $Q_{p^{*0}}(p^*)$ when $C$ grows.
He then proposes to subtract~$1$ from $G_{p^{*0}}(p^*)$ for each zero, that is $c$ in total.
He proves the asymptotic equivalence between $G_{p^{*0}}(p^*)$ and its corrected version only for $c=1$.
I will explain why his reasoning is inconsistent.
First, he considers a new statistic which is not a member of the power divergence family.
Moreover, he uses the straightforward Lemma \ref{lem1} to deduce the approximation:
\begin{equation} \label{approx}
2 n_r \ln \left( \frac{n_r}{n p_r^{*0}} \right) \simeq \frac{n_r ^2 - (n p_r^{*0})^2}{n p_r^{*0}},
\end{equation}
for $a=n_r/n$, $b= p_r^{*0}$ and $n_r=0$, despite the fact that $a=0$ and $1/a$ is not bounded.

\begin{lem} \label{lem1}
For each $a, b >0$ such that $a<2b$, $b<2a$ and the quantities $1/a$ and $1/b$ are bounded, we have :
$$ \ln \left( \frac{a}{b} \right) = \frac{a^2-b^2}{2ab}+o(a-b).$$ 
\end{lem}

The expression on the left in \eqref{approx} is null whereas the one on the right is negative.
The sum over $r$ of the left-hand side gives $G_{p^{*0}}(p^*)$ and we recognize $Q_{p^{*0}}(p^*)$ in the sum of the right-hand side.
However, unlike Ku seems to think, zero and non-zero cells tend to compensate for each other and the approximation \eqref{approx} can not be extended to the corresponding sum.
There is therefore no behavior of $G$ and $Q$ that we can deduce from this.
Finally, he illustrates this asymptotic result on a small sample of size $n=10$.

We propose new corrections for both statistics, based on Ku's correction and a likelihood inequality that we present in the next subsection.

\subsection{Likelihood inequality}

Let us consider the following inequality coming from a likelihood reasoning : the sample vector we observe can be thought of more likely to happen than any other possible vector, since it is the one we actually observed.
With exactly $c$ zeros and $R-c$ non-zero cells observed, so for all $m \leqslant c$ and $n'_j$ such that $\displaystyle n'_j \leqslant n_j, \ \forall j \in \{c+1,\ldots,R\}$ and 
$ n=\sum_{i=1}^{m} n'_i + \sum_{j=c+1}^{R}n'_j $ :
\begin{multline} \label{likely}
\PP(N_1=0,\ldots,N_c=0,N_{c+1}=n_{c+1},\ldots,N_R=n_R) \geqslant \\
  \PP(N_1=n'_1,\ldots, N_m=n'_m,N_{m+1}=\ldots=N_c=0,N_{c+1}=n'_{c+1},\ldots,N_R=n'_R).
\end{multline}

We give in Proposition \ref{condition} a sufficient condition on the $p_r$ for \eqref{likely} to be satisfied, and prove this statement in Appendix \ref{app}.

\begin{prop} \label{condition}
The inequality \eqref{likely} is satisfied under the following assumption : 
\begin{equation} \label{cond2}
p_i \leqslant \frac{ p_j}{n},\quad \forall i \in \{1,\ldots,c\}, \quad \forall j \in \{c+1,\ldots,R\}.
\end{equation}
\end{prop}

\subsection{Corrections}

We propose an estimator $\hat p$ for $p$, different from the maximum likelihood estimator, with the following form :
\begin{equation}
\begin{cases}
\displaystyle \hat p_i &= a, \quad \quad \quad \quad \quad \quad \quad \ \ \forall i \in \{1,\ldots,c \}, \\
\displaystyle \hat p_j &= \displaystyle \frac{n_j}{n^b}-d,  \quad \quad \quad \quad \quad \: \forall j \in \{c+1,\ldots,R \},\\
\end{cases}
\end{equation}
where $a=a_n,\ b=b_n$ and $d=d_n$ are random variables depending on $n$ designed to compensate for the under- and overestimations due to $p^*$.
We thus take them positive with $0 < b <1$, such that $b$ inflates the modified maximum likelihood estimator $n_j/n^b$ and such that $d$ controls the related rise.

This new probability vector allows us to define corrected statistics $Q_{p^{*0}}(\hat p)$ and $G_{p^{*0}}(\hat p)$, provided that the parameters $a,\ b$ and $d$ satisfy several conditions.
Forcing the summation of the $\hat p_r$ to $1$ implies that
%\begin{equation}
$(R-c)d=ac+n^{1-b}-1$,
%\end{equation}
and thus allows us to define $Q_{p^{*0}}(\hat p^{ab})$ and $G_{p^{*0}}(\hat p^{ab})$ with $\hat p^{ab}$ such that :
\begin{defi}
\begin{equation}
\label{estphat}
\begin{cases}
\displaystyle \hat p_i^{ab} &= a, \quad \quad \quad \quad \quad \quad \quad \quad \quad \forall i \in \{1,\ldots,c \}, \\
\displaystyle \hat p_j^{ab} &= \displaystyle \frac{n_j}{n^b}-\frac{ac+n^{1-b}-1}{R-c}, \quad \: \forall j \in \{c+1,\ldots,R \}.\\
\end{cases}
\end{equation}
\end{defi}
Note that for $c=0$, we fix $a=0$ and $b=1$.

Before considering the other conditions on $\hat p^{ab}$, let us first give some notations.
Let $\underline n$ be $\min_{j \in \{c+1,\ldots,R\}} \{n_j\}$ and $\overline n$ be $\max_{j \in \{c+1,\ldots,R\}} \{n_j\}.$
Let $\underline{\underline{n}}$ stand for $n-\underline n (R-c)$ and $\overline{\overline{n}}$ for $\overline n (R-c)-n$.
We dismiss the uniformly distributed case where :
\begin{equation}
\underline n = \overline n=n_j = \frac{n}{R-c}, \: \forall j \in \{c+1,\ldots,R\},
\end{equation}
thus guaranteeing that $\underline{\underline{n}}$ and $\overline{\overline{n}}$ are positive.
Let 
\begin{equation}
b_{\text{min}}=\max \left(0,\frac{\ln \left( \overline{\overline{n}} /(R-1)\right)}{\ln(n)},\frac{\ln (\underline{\underline{n}})}{\ln(n)},\frac{\ln \left( \overline n - \underline n \right)}{\ln(n)}\right),
\end{equation}
\begin{equation}
a_{\text{min}}(b)=\max \left(0,\frac{(\overline n-n^b)(R-c)+n^b}{c n^b}\right),
\end{equation}
and 
\begin{equation}
a_{\text{max}}(b)=\min \left(1,\frac{n^b-\underline{\underline{n}}}{c n^b},\frac{n^b-\underline{\underline{n}}}{n^b (n(R-c)+c)}\right),
\end{equation}
Proposition \ref{condition} is applied to $\hat p^{ab}.$
Together with inequalities $\displaystyle 0<\hat p_r^{ab} <1$, for all $r$ in $\{1,\ldots,R\}$ it is then equivalent to these conditions on $a$ and $b$:
\begin{equation} \label{interb}
b_{\text{min}}<b<b_{\text{max}}=1,
\end{equation}
and 
\begin{equation} \label{intera}
a_{\text{min}}(b)<a< a_{\text{max}}(b).
\end{equation}

We want to make a practical choice among possible values of $\hat p^{ab}$ ensuring us that it is as far from $p^*$ as possible.
We therefore fix $b$ in \eqref{interb} quite far from $1$ as a convex combination of $b_{\text{min}}$ and $b_{\text{max}}$ with an empirical parameter $h$ equal to $0.1$.
%(after several trials)  
For this value of $b$ we choose $a$ near the upper limit of the interval in \eqref{intera}, that is :
\begin{equation}
b=h b_{\text{max}}+(1-h) b_{\text{min}} \quad \text{ and } \quad a=a_{\text{max}}(b)-\epsilon,
\end{equation}
where $\epsilon$ is a small constant designed to eliminate boundary effects.

The final expressions we get for the corrected statistics $Q^{ab}=Q_{p^{*0}}(\hat p^{ab})$ and $G^{ab}=G_{p^{*0}}(\hat p^{ab})$ of $Q=Q_{p^{*0}}(p^*)$ and $G=G_{p^{*0}}(p^*)$ are:
%Detailed expressions for $Q^{ab}$ and $G^{ab}$ :
\begin{equation}
Q^{ab}= n^{2(1-b)} Q-f(a,b),
\end{equation}
with
\begin{multline}
f(a,b)=n \left(1-n^{2(1-b)}+ \frac{2 n^{1-b}(ac+n^{1-b}-1)}{R-c} \sum_{j=c+1}^{R} \frac{n_j}{n p^{*0}_j} \right.\\
\left.-a^2 \sum_{i=1}^{c} \frac{1}{p^{*0}_i}-\left(\frac{ac+n^{1-b}-1}{R-c}\right)^2 \sum_{j=c+1}^{R} \frac{1}{p^{*0}_j} \right),
\end{multline}
and
\begin{equation}
G^{ab}=n^{1-b} G-g(a,b),
\end{equation}
where
\begin{multline}
g(a,b)= 2n \left( \frac{ac+n^{1-b}-1}{R-c} \sum_{j=c+1}^{R} \ln \left(\frac{n_j (R-c)- n^b(ac+n^{1-b}-1)}{p^{*0}_j n^b (R-c)} \right) \right. \\
\left. -a\sum_{i=1}^{c} \ln \left( \frac{a}{p_i^{*0}} \right)-n^{1-b} \sum_{j=c+1}^{R} \frac{n_j}{n} \ln \left( \frac{n_j (R-c)-n^b(ac+n^{1-b}-1)}{n_j n^{b-1} (R-c)} \right) \right).
\end{multline}

\subsection{Convergence}

In this paragraph, we show the convergence in distribution of $Q^{ab}$ and $G^{ab}$ to a chi-square distribution.
Let us first study the parameter $b$.

\begin{propriete} \label{propb}
Bound $b_{\text{min}}$ is strictly less than $1$.
Moreover, if $\underline n=o(n)$ and $\overline n \sim n$ when $n$ tends to $+ \infty$, then $b_{\text{min}}$ and $b$ tend to $1$.
\end{propriete}

\begin{proof}
Inequalities $\overline n (R-c) < n R$, $\: n-\underline n (R-c)<n$ and $\overline n-\underline n<n$ show that all three components of the maximum defining $b_{\text{min}}$ are strictly less than~$1$. 
Their order $1$ developments as $n$ tends to $+ \infty$ give their convergence to~$1$, with the quantities $R-1$, $R-c$ and $R-c-1$ bounded.
\end{proof}

We also study the asymptotic behavior of $C$, for once denoted $C_n$, and state the following lemma which proof appears in Appendix \ref{app}.

\begin{lem} \label{lem2}
The number of zeros $C_n$ converges almost surely to $0$ as $n$ tends to $+ \infty$.
\end{lem}

%Assume now that $p^{*0}$ converges in probability to $p^0(\theta)$.
%\begin{equation}
%p^{*0} \xrightarrow{\PP} p^0(\theta).
%\end{equation}
We recalled in section \ref{sec1} the convergence of Pearson's and Kullback's statistics to a chi-square distribution.
In Theorem \ref{thm7}, we state a similar property for the corrected statistics $Q^{ab}$ and $G^{ab}$.

\begin{thm} \label{thm7}
Under Birch's regularity conditions in \cite{r11}, the estimator $\hat p^{ab}$ defined by \eqref{estphat}, \eqref{interb} and \eqref{intera} is such that :
\begin{equation}
\lim_{n \to +\infty} \mathcal L_{{\cal H}_0} (Q^{ab})=\lim_{n \to +\infty} \mathcal L_{{\cal H}_0} (G^{ab})=\chi^2_{R-s-1}.
\end{equation}
\end{thm}

\begin{proof}
We deduce from Lemma \ref{lem2} the existence of a set $\Omega'$ of probability~$1$ on which we can find a rank $n_0$ such that $C_n=0$ for $n \geqslant n_0$.
The variable~$a$ is then set equal to $0$ and the variable $b$ equal to $1$. 
Hence, the estimates $p^*$ and $\hat p^{ab}$ match, so that $Q^{ab}=Q$ and $G^{ab}=G$.
Theorem \ref{asymp} then completes the proof.
\end{proof}

The final two sections are dedicated to proving the relevancy of our corrections through simulations and real data analyses.

\section{Simulations}

In this section, simulations confirm the necessity to correct not only $G$ but also $Q$.
We compute the statistics $Q$, $G$, $RC^{2/3}=RC^{2/3}_{p^{*0}}(p^*)$, $Q^{ab}$ and $G^{ab}$ on $1000$ vectors of length $R=100$ of total frequency $n=400$ for each of the four multinomial distributions defined in Table \ref{Type1} by $f_1$ to $f_4$.  
Let $E_r$ denote the expected frequencies for $r$ in $\{1,\ldots,100\}$.

\begin{table}
\centering
\caption{Multinomial probabilities $f_1$ to $f_4$.\label{Type1}}
\begin{tabular}{ccc}
\hline
${\cal H}_0$ & $|\{r; \: E_r<0.5\}|$ & Probability \\
\hline
$f_1$ & $20$ & $( \underbrace{0.0002,\ldots,0.0002}_{20 \text{ times }},\underbrace{0.01245,\ldots,0.01245}_{80 \text{ times }} )$ \\
$f_2$ & $50$ & $( \underbrace{0.0002,\ldots,0.0002}_{50 \text{ times }},\underbrace{0.0198,\ldots,0.0198}_{50 \text{ times }} )$ \\
$f_3$ & $70$ & $( \underbrace{0.0002,\ldots,0.0002}_{70 \text{ times }},\underbrace{0.03286667,\ldots,0.03286667}_{30 \text{ times }})$ \\
$f_4$ & $90$ & $( \underbrace{0.0002,\ldots,0.0002}_{90 \text{ times }},\underbrace{0.0982,\ldots,0.0982}_{10 \text{ times }} )$ \\
\hline
\end{tabular}
\end{table}

For each set of vectors sharing the same $c$, we give the quantile of order $1-\alpha=95\%$ for the five statistics considered.
Results are displayed in figure~\ref{CQGetc1}, as well as a line indicating the chi-square quantile of order $1-\alpha=95\%$ with $R-1=99$ degrees of freedom $\chi_{0.95,99}^2=123.22$.
Only the center of each graph should be considered because the quantiles for extreme numbers of zeros are computed on too few observations, sometimes only on one or two among the $1000$ simulations in total.

\begin{figure}
\begin{center}
%\vspace{-0.75cm}
%\includegraphics[width=13cm,height=9cm]{Graphes22-coef10-simu2-alpha5.eps}
\includegraphics[width=12cm,height=18cm]{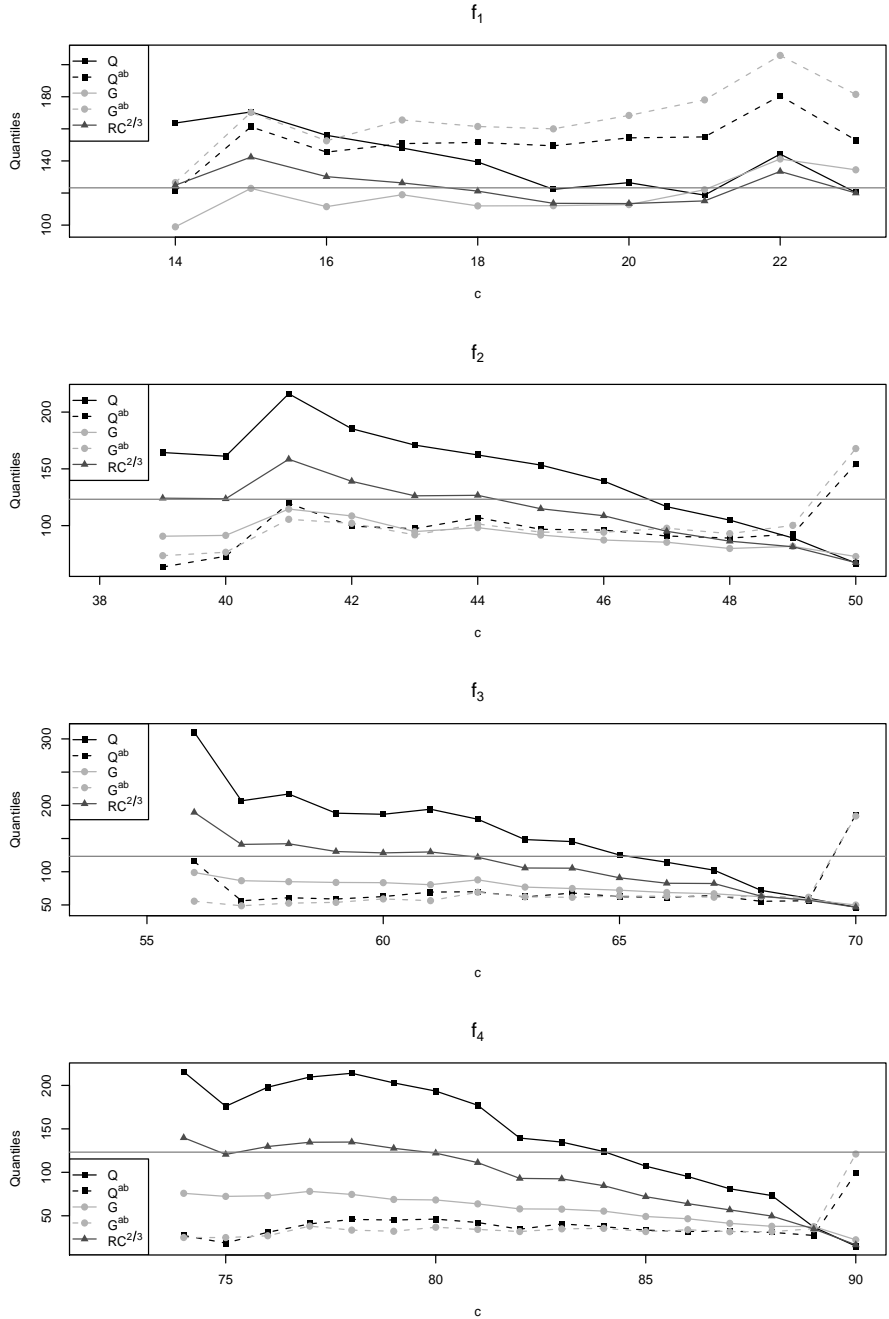}
\caption{Quantiles of order $0.95$ for $Q$, $Q^{ab}$, $G$, $G^{ab}$ and $RC^{2/3}$ as functions of $c$, under null multinomial probabilities $f_1$ to $f_4$, for $1000$ samples of size $n=400$ and $R=100$ categories. The line represents the threshold $\chi_{0.95,99}^2$.\label{CQGetc1}}
\end{center}
\end{figure}

%\begin{figure}[htbp]
%\begin{center}
%%\fbox{
%\vspace{-0.75cm}
%\includegraphics[width=15cm,height=20cm]{CQGetc1beau.eps}
%%}
%%\caption{Quantiles d'ordre $0.95$ de $Q$, $Q^{ab}$, $G$, $G^{ab}$, $G^{\text{Ku}}$ et $RC^{2/3}$ en fonction de $c$, sous les hypothèses nulles $f_1$ à $f_3$, pour $1000$ échantillons de $n=400$ individus et $R=100$ catégories. La ligne horizontale représente le seuil critique $\chi_{0.95,99}^2$.\label{CQGetc1}}
%\end{center}
%\end{figure}
%
%\begin{figure}[htbp]
%\begin{center}
%%\fbox{
%\vspace{-0.75cm}
%\includegraphics[width=15cm,height=20cm]{CQGetc2beau.eps}
%%}
%%\caption{Quantiles d'ordre $0.95$ de $Q$, $Q^{ab}$, $G$, $G^{ab}$, $G^{\text{Ku}}$ et $RC^{2/3}$ en fonction de $c$, sous les hypothèses nulles $f_4$ à $f_6$, pour $1000$ échantillons de $n=400$ individus et $R=100$ catégories. La ligne horizontale représente le seuil critique $\chi_{0.95,99}^2$.\label{CQGetc2}}
%\end{center}
%\end{figure}

Quantile values of $Q$ tend to explode as $c$ grows whereas $G$ stays quite stable around $\chi_{0.95,99}^2$.
This behavior is the opposite of the one predicted by Ku.
For $f_1$, statistics $Q^{ab}$ and $G^{ab}$ lead to the rejection of the null hypothesis.
For $f_2$ to $f_4$ however, their quantiles lie below the critical line and ${\cal H}_0$ is accepted.
We thus have compensated for the rise of $Q$, and both our corrected statistics are stable.

\begin{table}
\centering
\caption{Empirical type I risks for $Q$, $Q^{ab}$, $G$, $G^{ab}$ and $RC^{2/3}$, for $1000$ samples of size $n=400$ and $R=100$ categories, at levels $\alpha=0.01$, $0.05$, $0.1$, for vector probabilities $f_1$ to $f_4$.\label{Type1ter}}
\begin{tabular}{cccccccc}
\hline
$\alpha$ & ${\cal H}_0$ & $mode(c)$ & $Q$ & $Q^{ab}$ & $G$ & $G^{ab}$ & $RC^{2/3}$ \\
\hline
$0.01$ & $f_1$ & $19$ & $0.031$ & $0.233$ & $0.003$ & $0.401$ & $0.003$ \\
 & $f_2$ & $47$ & $0.030$ & $0.005$ & $0$ & $0$ & $0$ \\
 & $f_3$ & $65$ & $0.027$ & $0$ & $0$ & $0$ & $0$ \\
 & $f_4$ & $84$ & $0.031$ & $0$ & $0$ & $0$ & $0$ \\
\\
$0.05$ & $f_1$ & $19$ & $0.044$ & $0.352$ & $0.010$ & $0.573$ & $0.017$ \\
 & $f_2$ & $47$ & $0.024$ & $0.005$ & $0$ & $0$ & $0$ \\
 & $f_3$ & $65$ & $0.069$ & $0$ & $0$ & $0$ & $0.006$ \\
 & $f_4$ & $84$ & $0.052$ & $0$ & $0$ & $0$ & $0$ \\
\\
$0.1$ & $f_1$ & $19$ & $0.130$ & $0.479$ & $0.033$ & $0.674$ & $0.039$ \\
 & $f_2$ & $47$ & $0.072$ & $0.005$ & $0$ & $0.005$ & $0$ \\
 & $f_3$ & $65$ & $0.137$ & $0$ & $0$ & $0$ & $0$ \\
 & $f_4$ & $84$ & $0.098$ & $0$ & $0$ & $0$ & $0.006$ \\
\hline
\end{tabular}
\end{table}

This analysis is confirmed by the computation of empirical risks of type~I for $\alpha=0.01$, $0.05$ and $0.1$ as showed in Table \ref{Type1ter}, and by the power study below.
Probabilities $f_1$ to $f_4$ are perturbed into $f'_1$ to $f'_4$ such that for all $j$ in $\{1,2,3,4\}$:
\begin{eqnarray}
\forall i \in \{1,\ldots,10\}, \: f'_j(i)&=&f_j(i)+1/300,\\
\forall i \in \{11,\ldots,90\}, \: f'_j(i)&=&f_j(i),\\
\forall i \in \{91,\ldots,100\}, \: f'_j(i)&=&f_j(i)-1/300.
\end{eqnarray}

Vectors are simulated with probabilities $f_1$ to $f_4$, and goodness of fit for $f'_1$ to $f'_4$ is tested.
The Tables \ref{Type1ter} and \ref{tab2} show that the empirical type I risks are lower for our corrections compared to the classical statistics when $c$ is quite important, whereas the empirical power is much higher for our corrections when $c$ is small.

\begin{table}
\centering
\caption{Empirical powers for $Q$, $Q^{ab}$, $G$, $G^{ab}$ and $RC^{2/3}$, for $1000$ samples of size $n=400$ and $R=100$ categories, at levels $\alpha=0.05$ for simulated vector probabilities $f_1$ to $f_4$, and null vector probabilities $f'_1$ to $f'_4$.\label{tab2}}
\begin{tabular}{cccccccc}
\hline
 ${\cal H}_0$ & ${\cal H}_1$ & $mode(c)$ & $Q$ & $Q^{ab}$ & $G$ & $G^{ab}$ & $RC^{2/3}$ \\
\hline
$f_1$ & $f'_1$ & $19$ & $0.233$ & $0.853$ & $0.322$ & $0.983$ & $0.157$ \\
$f_2$ & $f'_2$ & $47$ & $0.087$ & $0.026$ & $0.009$ & $0.061$ & $0.009$ \\
$f_3$ & $f'_3$ & $64$ & $0.229$ & $0.005$ & $0$ & $0$ & $0.027$ \\
$f_4$ & $f'_4$ & $84$ & $0.089$ & $0$ & $0$ & $0$ & $0$ \\
\hline
\end{tabular}
\end{table}

\section{Applications}

We apply the total independence test using the corrected statistics $Q^{ab}$ and $G^{ab}$ to two datasets involving two-dimensional tables.
For such tables, the hypotheses are written :
\begin{equation}
{\cal H}_0 :\quad p_{ij}= p_{i +} p_{+ j} \quad \text{ against } \quad {\cal H}_1 :\quad \exists \ (i_1,j_1) \in I \times J, \: p_{i_1 j_1} \neq p_{i_1 +} p_{+ j_1},
\end{equation}
where $p_{i +}$ and $p_{+ j}$ are the marginal distributions for the two characters featured in the table.
To ensure the condition \eqref{nonnul} we remove the empty lines $i$ of $\{1,\ldots,I\}$ such that $n_{i+}=0$, and the empty columns $j$ of $\{1,\ldots,J\}$ such that $n_{+j}=0$.

\subsection{Multi-marker approach for Systemic Sclerosis}

Table \ref{diplo} is the diplotypic table obtained from an association study in Humans looking for an association between three genetic markers on the gene TNFAIP3 and Systemic Sclerosis presented in \cite{r8}.
Empty columns have been removed.
A haplotype is the allelic distribution of markers on a chromosome, and a diplotype is the combination of both parental haplotypes.
Diplotypic tables, though more interesting than haplotypic tables because they take into account more information, are usually trickier to handle because they are sparse. 
Our corrected statistics can therefore be helpful in such situations.

Though nine haplotypes theoretically exist, only eight are observed, denoted H1 to H8, leading to $8^2=64$ diplotypes Hi/Hj.
Two samples are compared, affected versus sound subjects, on which we test the independence between the diplotype configuration and the health status of $n=794$ individuals.
\begin{table}
\centering
\caption{Diplotype table for the association between TNFAIP3 and Systemic Sclerosis..\label{diplo}}
\begin{tabular}{ccccccc}
\hline
%\backslashbox{Statut}{Diplotype} & H1/H1 & H1/H2 & H1/H3 & H1/H4 & H1/H5 & H1/H6\\
Status & H1/H1 & H1/H2 & H1/H3 & H1/H4 & H1/H5 & H1/H6 \\
\hline
Sound & $98$ & $7$ & $116$ & $2$ & $71$ & $3$   \\ 
Affected & $91$ & $9$ & $104$ & $3$ & $70$ & $12$ \\ 
%\hline \hline
\\
& H2/H3 & H2/H5 & H2/H6 & H3/H3 & H3/H4 & H3/H5 \\ 
\cline{2-7}
Sound & $4$ & $2$ & $0$ & $34$ & $1$ & $42$ \\
Affected & $5$ & $4$ & $1$ & $30$ & $2$ & $40$ \\
\\
& H3/H6 & H4/H5 & H5/H5 & H5/H6 & &\\ 
\cline{2-7}
Sound & $2$ & $1$ & $13$ & $1$ & &\\
Affected & $7$ & $1$ & $13$ & $5$ & & \\
\hline
\end{tabular}
\end{table}

%\begin{table}
%\centering
%\caption{Diplotype table for the association between TNFAIP3 and Systemic Sclerosis.\label{diplo}}
%\begin{tabular}{ccccccccc}
%\hline
%Status & H1/H1 & H1/H2 & H1/H3 & H1/H4 & H1/H5 & H1/H6 & H2/H1 & H2/H3\\
%\hline
%Sound & $98$ & $7$ & $116$ & $2$ & $71$ & $3$ & $7$ & $4$ \\ 
%Affected & $91$ & $9$ & $104$ & $3$ & $70$ & $12$ & $9$ & $5$ \\ 
%\\
% & H2/H5 & H2/H6 & H3/H1 & H3/H2& H3/H3 & H3/H4 & H3/H5 & H3/H6\\ 
%\hline
%Sound  & $2$ & $0$ & $116$ & $4$& $34$ & $1$ & $42$ & $2$\\
%Affected & $4$ & $1$ & $104$ & $5$& $30$ & $2$ & $40$ & $7$\\
%\\
% & H4/H1 & H4/H3 & H4/H5 & H5/H1 & H5/H2 & H5/H3 & H5/H4 & H5/H5\\
%\hline
%Sound  & $2$ & $1$& $1$ & $71$ & $2$ & $42$ & $1$ & $13$\\
%Affected & $3$ & $2$& $1$& $70$ & $4$ & $40$ & $1$ & $13$\\
%\\
% & H5/H6 & H6/H1 & H6/H2 & H6/H3 & H6/H5 &&&\\
% \hline
% Sound& $1$ & $3$ & $0$ & $2$& $1$ &&&\\
% Affected& $5$ & $12$ & $1$ & $7$ & $5$&&&\\
%\hline
%\end{tabular}
%\end{table}
The table is of dimension $2 \times 16$, that is $R=32$ categories with $c=1$ zero and $s=16$ parameters.
There are exactly $16$ expected frequencies below $5$.
The chi-square quantile $\chi^2_{0.95,15} = 24.99$ is compared to the statistics :
$$
Q=14.62, Q^{ab}=20.76, G=15.82, G^{ab}=28.43, RC^{2/3}=14.85.
$$
Only $G^{ab}$ leads to reject the null hypothesis of independence.
This seems to be the right decision since it is confirmed by the single-markers approaches and haplotype tests in \cite{r8}, all showing a significative association between the markers and the disease.

\subsection{Trophic level and vegetables in the rivers of the Petite Camargue Alsacienne}

The search for a link between the trophic level and the vegetable composition of some rivers of the Petite Camargue Alsacienne in North-East France leads to Table \ref{Troph}.
A river can be either oligotrophic, mesotrophic or eutrophic, if its nutritive content is respectively poor, intermediate or high.
Uncommon vegetables are considered rare, exotic or polluto-tolerant.
To each river, a triplet of binary characteristics $(r,p,e)$ is assigned, indicating the presence ($1$) or the absence ($0$) of rare ($r$), exotic ($e$) and polluto-tolerant~($p$) species.
The original ecological study can be found in \cite{r9}.
We consider $n=21$ different rivers.
Two empty columns were removed from the original table, leading to Table \ref{Troph} of dimension $3 \times 6$, that is $R=18$ categories and $c=7$ zeros, with $s=17$ parameters.
Here are $3$ expected frequencies below~$0.5$.

\begin{table}
\centering
\caption{Contingency table for the joint study of trophic level and vegetable composition in rivers.\label{Troph}}
\begin{tabular}{ccccccc}
\hline
& \multicolumn{6}{c}{$(r,p,e)$}\\
Trophic level & $(0,0,0)$ & $(1,0,0)$ & $(0,1,0)$ & $(0,0,1)$ & $(1,1,0)$ & $(0,1,1)$\\
\hline
Oligotrophic & $0$ & $0$ & $3$ & $0$ & $3$ & $2$  \\
Mesotrophic & $2$ & $1$ & $0$ & $2$ & $1$ & $0$  \\ 
Eutrophic & $2$ & $0$ & $3$ & $1$ & $1$ & $0$  \\ 
\hline
\end{tabular}
\end{table}

Test statistics are compared to the chi-square quantile $\chi^2_{0.95,10} = 18.31$:
$$
Q=14.38, Q^{ab}=20.68, G=18.67, G^{ab}=26.05, RC^{2/3}=14.84.
$$
Both corrected statistics $Q^{ab}$ and $G^{ab}$ as well as $G$ lead to reject the null hypothesis, indicating an association between trophic level and vegetable composition.
A thorough study of the table shows that rare species tend to settle preferentially in oligotrophic rivers.
They are indeed better adapted to this kind of environment which tends to disappear from the rivers in the study.
Moreover, polluto-tolerant species constitute the majority of the vegetables in eutrophic rivers.
A eutrophic environment is competitive and these resistant species tend to get the best of it.

\subsection{Discussion}
We suggest to compute both $Q^{ab}$ and $G^{ab}$, and to reject the null hypothesis if at least one of them is larger than the chi-square threshold.

Our results tend to prove that this approach is relevant and our corrections efficient in sparse tables.
They are all the more interesting for the fact that sparse tables are usually left aside because the hypotheses needed to apply classical chi-square tests are not satisfied.

\appendix

\section{Appendix section}\label{app}

\begin{proof}[Proof of Proposition \ref{condition}]
Assume that the first $m$ of the $c$ frequencies $n_i, 1 \leqslant i \leqslant c$ are modified.
Let $n'_j, j \in \{c+1,\ldots,R\}$ compensate for these changes.
The likelihood inequality \eqref{likely} is then equivalent to :
\begin{equation} \label{lab2}
\frac{p_{c+1}^{(n_{c+1}-n'_{c+1})}}{(n_{c+1}-n'_{c+1}+1)!} \times \cdots \times \frac{p_{R}^{(n_{R}-n'_{R})}}{(n_{R}-n'_{R}+1)!} \geqslant \frac{p_1^{n'_1}}{n'_1!} \dots \frac{p_m^{n'_m}}{n'_m!}.
\end{equation}
Let us show that \eqref{cond2} implies \eqref{lab2}.
Applying \eqref{cond2} to each element of the left hand side of \eqref{lab2} with multiplicities such that :
$\displaystyle \sum_{j=c+1}^{R} (n_j-n'_j) = \sum_{i=1}^{m} n'_i,$
we get : 
\begin{equation}
\frac{p_{c+1}^{(n_{c+1}-n'_{c+1})}}{n^{(n_{c+1}-n'_{c+1}+1)}} \times \cdots \times \frac{p_{R}^{(n_{R}-n'_{R})}}{n^{(n_{R}-n'_{R}+1)}} \geqslant
p_1^{n'_1} \dots p_m^{n'_m}.
\end{equation}
As $(n_j-n'_j+1)! \leqslant n^{(n_j-n'_j)}$ for all $j$ in $\{c+1,\ldots,R\}$ and
$n'_i! \geqslant 1$ for all $i$ in $\{1,\ldots,m\},$
we deduce \eqref{lab2} and equivalently \eqref{likely}.
\end{proof}

\begin{proof}[Proof of Lemma \ref{lem2}]
Let us first show that:
\begin{equation}
\forall \epsilon >0, \lim_{n \to +\infty} \PP (C_n>\epsilon) =0.
\end{equation}
For each $n$ we compute $\PP (C_n=c)$ for $c$ in $\{0,\ldots,R-1\}.$
Let $p^0$ be the probability under the null hypothesis $p^0=(p^0_1,\ldots,p^0_c,\ p^0_{c+1},\ldots,p^0_R)$.
A subject belongs to one of the first $c$ cells with probability $q^0=p^0_1 + \dots + p^0_c$ and to one of the $R-c$ last cells with probability $1-q^0=p^0_{c+1} + \dots + p^0_R$, with $q^0$ in $]0,1[$. 

Let us consider now the binomial distribution $\mathcal B(n;q^0)$ and write the probability $\PP (C_n=c)$ of obtaining a table containing exactly $c$ zeros placed anywhere :
\begin{eqnarray}
\quad \quad \PP (C_n=c) &=& \binom{R}{c} \PP (N_1=\ldots=N_c=0,\ N_{c+1} \neq 0,\ldots,N_R \neq 0),\\
&=&\binom{R}{c} (1-q^0)^n.
\end{eqnarray}
Then :
\begin{equation}
\PP (C_n=0)= 1- \sum_{c=1}^{R-1} \PP(C_n=c)=1-\PP(C_n>\epsilon), \quad \forall \epsilon \in ]0,1[.
\end{equation}
As $\binom{R}{c}$ is bounded and $(1-q^0)^n$ tends to $0$, the probability $\PP(C_n=c)$ also converges to $0$ for all $1 \leqslant c \leqslant R-1$ when $n$ tends to $+\infty$, and so does the corresponding sum over $c$.

We now use Borel-Cantelli's Lemma to conclude that $C_n$ converges to $0$ almost surely.
Indeed, for $\epsilon >0$:
\begin{equation}
\sum_{n \geqslant 1} \PP(C_n > \epsilon) \leqslant \sum_{n \geqslant 1} \PP(C_n \geqslant 1)=\sum_{c=1}^{R-1} \binom{R}{c} \frac{1}{q^0} < + \infty.
\end{equation}
\end{proof}

\section*{Acknowledgements}
I want to thank I. Combroux and M. Guedj for letting me use their data sets as illustrations for this work.
I am also truly grateful to P. Nobelis for his help and his advice.


\begin{thebibliography}{9}
\bibitem{r1}
\textsc{Pearson, K.} (1900). 
On the criterion that a given system of deviations from the probable in the case of a correlated system of variables is such that it can be reasonably supposed to have arisen from random sampling.
\textit{Philosophical Magazine}
\textbf{50} 157--175.
\end{thebibliography}
\end{document}